\documentclass[12pt]{amsproc}
\usepackage[pdftex]{hyperref}
\newtheorem{theorem}{\sc Theorem}[section]
\newtheorem{lemma}[theorem]{\sc Lemma}

\newtheorem{corollary}[theorem]{\sc Corollary}

\newtheorem{hypothesis}[theorem]{\sc Hypothesis}

\begin{document}
\title[Neumann's BFC-theorem]{A stronger form of Neumann's BFC-theorem}
\author{Cristina Acciarri}

\address{Cristina Acciarri:  Department of Mathematics, University of Brasilia,
Brasilia-DF, 70910-900 Brazil}
\email{acciarricristina@yahoo.it}

\author{Pavel Shumyatsky }
\address{ Pavel Shumyatsky: Department of Mathematics, University of Brasilia,
Brasilia-DF, 70910-900 Brazil}
\email{pavel@unb.br}
\thanks{This research was supported by the Conselho Nacional de Desenvolvimento Cient\'{\i}fico e Tecnol\'ogico (CNPq),  and Funda\c c\~ao de Apoio \`a Pesquisa do Distrito Federal (FAPDF), Brazil.}
\dedicatory{Dedicated to Andrea Lucchini on the occasion of his 60th birthday}

\begin{abstract} Given a group $G$, we write $x^G$ for the conjugacy class of $G$ containing the element $x$. A famous theorem of B. H. Neumann states that if $G$ is a group in which all conjugacy classes are finite with bounded size, then the derived group $G'$ is finite. We establish the following result.

\noindent Let $n$ be a positive integer and $K$ a subgroup of a group $G$ such that $|x^G|\leq n$ for each $x\in K$. Let $H=\langle K^G\rangle$ be the normal closure of $K$. Then the order of the derived group $H'$ is finite and $n$-bounded. 

\noindent Some corollaries of this result are also discussed.
\end{abstract}

\maketitle

\section{Introduction} Given a group $G$ and an element $x\in G$, we write $x^G$ for the conjugacy class containing $x$. Of course, if the number of elements in $x^G$ is finite, we have $|x^G|=[G:C_G(x)]$. A group is called a BFC-group if its conjugacy classes are finite and have bounded size. In 1954 B. H. Neumann discovered that in a BFC-group the derived group $G'$ is finite \cite{bhn}. It follows that if $|x^G|\leq n$ for each $x\in G$, then the order of $G'$ is bounded by a number depending only on $n$. A first explicit bound for the order of $G'$ was found by J. Wiegold \cite{wie}, and the best known was obtained in \cite{gumaroti} (see also \cite{neuvoe} and \cite{sesha}). The article \cite{dieshu} deals with groups $G$ in which conjugacy classes containing commutators are bounded. In particular, \cite{dieshu} contains a proof of the following theorem.
\bigskip

{\it Let $n$ be a positive integer and $G$ a group. If $|x^G|\leq n$ for any commutator $x$, then $|G''|$ is finite and $n$-bounded. }
\bigskip

Here  $G''$ denotes the second commutator subgroup of $G$. Throughout the article we use the expression ``$(a,b,\dots)$-bounded" to mean that a quantity is finite and bounded by a certain number depending only on the parameters $a,b,\dots$. The following extension of the aforementioned theorems to higher commutator subgroups was obtained in \cite{dms}.

\bigskip

{\it Let $n$ be a positive integer and $w$ a multilinear commutator word. Suppose that $G$ is a group in which $|x^G|\leq n$ for each $w$-value $x\in G$. Then the verbal subgroup $w(G)$ has derived group of finite $n$-bounded order. }
\bigskip

A related result for groups in which the conjugacy classes containing squares have finite bounded sizes was established in \cite{squares}.

In the present paper we obtain a variation of different nature for the Neumann theorem.

\begin{theorem} \label{main}
Let $n$ be a positive integer,  $G$ a group  having a subgroup $K$ such that $|x^G|\leq n$ for each $x\in K$, and let $H=\langle K^G\rangle$. Then the order of the derived group $H'$ is finite and $n$-bounded. 
\end{theorem}

Here, as usual, $\langle K^G\rangle$ denotes the normal closure of $K$ in $G$.

Theorem \ref{main} will be proved in the next section. Section 3 contains several easy but surprising corollaries of Theorem \ref{main} for finite groups. In Section 4 we handle profinite groups with restricted centralizers. A group $G$ is said to have restricted centralizers if for each $g$ in $G$ the centralizer $C_G(g)$ either is finite or has finite index in $G$. This notion was introduced by Shalev in \cite{shalev} where he showed that a profinite group with restricted centralizers is virtually abelian. We say that a profinite group has a property virtually if it has an open subgroup with that property. The recent article \cite{dms2} handles profinite groups with restricted centralizers of $w$-values for a multilinear commutator word $w$. The theorem proved in \cite{dms2} says that if $w$ is a multilinear commutator word and $G$ is a profinite group in which the centralizer of any $w$-value is either finite or open, then the verbal subgroup $w(G)$ is virtually abelian. In Section 4 we establish

\begin{theorem} \label{restricted}
Let $p$ be a prime and $G$ a profinite group in which the centralizer of each $p$-element is either finite or open. Then $G$ has a normal abelian pro-$p$ subgroup $N$ such that $G/N$ is virtually pro-$p'$.
\end{theorem}

The proof of Theorem \ref{restricted} is based on Theorem \ref{main}. It also uses the celebrated theorem of Zelmanov \cite{zelmanov} which states that a torsion pro-$p$ group is locally nilpotent. 

\section{Proof of the main result} 

The purpose  of this section is to prove Theorem \ref{main}. We start with the following result which is somewhat similar to Dietzmann's lemma which says that  if $x$ is an element of finite order in a group $G$ such that $x^G$ is finite, then also the subgroup $\langle x^G \rangle$ is finite (see \cite{Rob2}).

\begin{lemma}\label{dits}
Let $i,j$ be  positive integers and $G$ a group having a subgroup $K$ such that $|x^G|\leq i$ for each $x\in K$. Suppose that $|K|\leq j$. Then $\langle K^G\rangle$ has finite $(i,j)$-bounded order. 
\end{lemma}
\begin{proof} Since $C_G(K)=\cap_{x\in K}C_G(x)$, we observe that the index of $C_G(K)$ in $G$ is at most $i^j$. Let $C$ be the maximal normal subgroup of $G$ contained in $C_G(K)$. Of course, the index of $C$ in $G$ is at most $(i^j)!$. Set $H=\langle K^G\rangle$. It follows that $|H/Z(H)|\leq (i^j)!$ and Schur's theorem \cite[10.1.4]{Rob} tells us that the order of $H'$ is $(i,j)$-bounded. Passing to the quotient $G/H'$ we can assume that $H$ is abelian. Then $H$ is a product of at most $i^j$ conjugates of $K$ and so the order of $H$ is at most $j^{i^j}$, as desired.
\end{proof}

Let $G$ be a group generated by a set $X$ such that $X = X^{-1}$. Given an element $g\in G$, we write $l_X(g)$ for the minimal number $l$ with the property that $g$ can be written as a product of $l$ elements of $X$.  The symbol $l_X(g)$ denotes the length of $g$ with respect to $X$. The proof of the following  result can be found in  \cite[Lemma 2.1]{dieshu}.

\begin{lemma}\label{21} Let $H$ be a group generated by a set $X = X^{-1}$ and let $K$ be a subgroup of finite index $m$ in $H$. Then each coset $Kb$ contains an element $g$ such that $l_X(g)\leq m-1$.
\end{lemma}

We now fix some notation and hypothesis  that we will use within  the current section.

\begin{hypothesis}\label{hypo} Let $G$ be a group  having a subgroup $K$ such that $|x^G|\leq n$, for any $x\in K$, and let $H=\langle K^G\rangle$. Denote by $X$ the set of all conjugates of elements of $K$, that is, $X=\bigcup_{g\in G} K^g$. Let $m$ be the maximum of indices of $C_H (x)$ in $H$ for $x\in K$. Select $a\in K$ such that $|a^H|=m$. Choose   $b_1,\ldots,b_m$ in $H$ such that $l_X(b_i)\leq m-1$ and $a^H=\{a^{b_i}; i=1,\ldots,m\}$. (The existence of the elements $b_i$ is guaranteed by Lemma \ref{21}.) Set  $U=C_G(\langle b_1,\ldots,b_m \rangle)$.
\end{hypothesis}

Some of the arguments we use in the proof of the next lemmas are similar to those used in \cite{dieshu}. 
\begin{lemma}\label{23} Assume Hypothesis \ref{hypo}. Then for any $x\in X$ the subgroup $[H, x]$ has finite
$m$-bounded order.
\end{lemma}

\begin{proof} Take $x \in X$. Since the index of $C_H(x)$ in $H$ is at most $m$, by Lemma \ref{21}, we can choose elements $y_1,\ldots,y_m$ in $H$ such that $l_X(y_i)\leq m-1$ and the subgroup $[H,x]$ is generated by the commutators $[y_i,x]$, for $i=1,\ldots,m$.  For any such $i$  write $y_i=y_{i1}\ldots y_{i(m-1)}$, with $y_{ij}\in X$. By using standard commutator identities we can rewrite $[y_i,x]$ as a product of conjugates in $H$ of the commutators $[y_{ij},x]$. Let $\{h_1,\ldots,h_s\}$ be the conjugates in $H$  of all elements from the set $\{x,y_{ij} \mid  1\leq i,j \leq m\}.$ Note that the number $s$ here is $m$-bounded. This follows form the fact that $C_H(x)$ has index at most $m$ in $H$ for each $x\in X$.  Put $T=\langle h_1,\ldots,h_s \rangle$. Since $[H,x]$ is contained in $T'$, it is enough to show that $T'$ has finite $m$-bounded order. Observe that the center $Z(T)$ has index at most $m^s$ in $T$, since the index of $C_H(h_i)$ is at most $m$ in $H$ for any $i=1,\ldots,s$.  Thus, by Schur's theorem \cite[10.1.4]{Rob}, we conclude that $T'$ has finite $m$-bounded order, as desired. 
\end{proof}

Note that the index of $U$ in $G$ is $n$-bounded. Indeed, since $l_{X}(b_i)\leq m-1$ we can write $b_i=b_{i1}\ldots b_{i(m-1)}$, where $b_{ij}\in X$ and $i=1,\ldots,m$. By the hypothesis the index of $C_G(b_{ij})$ in $G$ is at most $n$ for any such element $b_{ij}$. Thus, $[G:U]\leq n^{(m-1)m}$. 

The next result is somewhat analogous to Lemma 4.5 of Wiegold  in \cite{wie}.
\begin{lemma}\label{24} Assume Hypothesis \ref{hypo}. Suppose that $u\in U$ and $ua\in X$. Then $[H, u] \leq [H, a]$.
\end{lemma}
\begin{proof} Recall that $U=C_G(\langle b_1,\ldots,b_m \rangle)$. For each $i=1,\ldots,m$ we have $(ua)^{b_i}=ua^{b_i}$, since $u$ belongs to $U$.  We know that $ua \in X$, so, taking into account the hypothesis on the order of the conjugacy class of $ua$ in $H$, we deduce that $(ua)^H$ consists  exactly of the elements $ua^{b_i}$, for $i=1,\ldots,m$.   Thus, given an arbitrary element $h\in H$, there exists $b\in \{b_1,\ldots,b_m\}$ such that $(ua)^h=ua^b$ and so $u^ha^h=ua^b$. It follows that $[u,h]=a^ba^{-h}\in [H,a]$, and the result holds.
\end{proof}

We are now ready to embark on the proof of Theorem \ref{main}.

\begin{proof}[Proof of Theorem \ref{main}]

Recall that $G$ is a group having a subgroup $K$ such that $|x^G|\leq n$ for any  $x$ in $K$. Let $H=\langle K^G\rangle$.   We need to show that $|H'|$ is finite and $n$-bounded. 

As above, the symbol $X$ denotes  the set  of all conjugates of  elements of $K$. Thus $H=\langle X \rangle$. Let $m$ be the maximum of indices of $C_H (x)$ in $H$, where $x\in K$. Note that $m\leq n$. Choose $a\in K$ such that $C_H(a)$ has index precisely $m$ in $H$.  By Lemma \ref{21} we can choose $b_1,\ldots,b_m$ in $H$ such that $l_{X}(b_i)\leq m-1$ and $a^H=\{a^{b_i}; i=1,\ldots,m\}$. Set $U = C_G(\langle b_1,\ldots,b_m \rangle)$. Note that the index of $U$ in $G$ is $n$-bounded. 

By the hypothesis $a$ has at most $n$ conjugates in $G$, say $\{a^{g_1},\ldots,a^{g_n}\}$, that are elements of $X$. Let $T$ be the normal closure in $G$  of  the subgroup $[H,a]$, thus $T=[H,a^{g_1}]\cdots [H,a^{g_n}]$. By Lemma \ref{23}  each of these subgroups has $n$-bounded order. We conclude that $T$ has finite $n$-bounded order.  

 Let $K_0=K\cap U$. Note that for any $u\in K_0$ the product $ua$ belongs to $K$.  Therefore, by  Lemma \ref{24}, for any $u$ in $K_0$, the subgroup $[H,u]$ is contained in $[H,a]$.
 
 Since $T$ has finite $n$-bounded order, it is sufficient to show that the derived group of the quotient $H/T$ has finite $n$-bounded order. We pass now  to the quotient $G/T$ and for the  sake of simplicity  the images of $G$, $H, U, K$ and  $K_0$   will be denoted by the same symbols.  Note that the subgroup $K_0$ becomes central in $H$ modulo $T$.   Next we consider the quotient $G/Z(H)$. Since the image of $K_0$ in $G/Z(H)$ is trivial, we deduce that the subgroup $K$ has $n$-bounded order modulo $Z(H)$.  Indeed, in $G/Z(H)$ the order of the image of $K$ is at most the index of $U$ in $G$.  Now it follows from Lemma \ref{dits} that $H$ has $n$-bounded order modulo $Z(H)$. Thus,  by Schur's theorem \cite[10.1.4]{Rob}, we conclude that $|H'|$ is $n$-bounded, as desired.
\end{proof}

\section{Corollaries concerning finite groups}

In this section we will record several easy corollaries of Theorem \ref{main} concerning finite groups. If $\phi$ is an automorphism of a group $G$, then  the centralizer $C_G(\phi)$ is the subgroup formed by the elements $x\in G$ such that $x^\phi=x$. In the case where $C_G(\phi)=1$ the automorphism $\phi$ is called fixed-point-free. A famous result of Thompson \cite{tho} says that a finite group admitting a fixed-point-free automorphism of prime order is nilpotent. Higman proved that for each prime $p$ there exists a number $h=h(p)$ such that whenever a nilpotent group $G$ admits a fixed-point-free automorphism of order $p$, it follows that $G$ is nilpotent of class at most $h$ \cite{hi}. Therefore the nilpotency class of a finite group admitting a fixed-point-free automorphism of order $p$ is at most $h$.

If $\phi$ is an automorphism of a finite group $G$ such that $(|G|,|\phi|)=1$, then for each normal $\phi$-invariant subgroup $N$ of $G$ we have $C_{G/N}(\phi)=C_G(\phi)N/N$ (see for example \cite[Theorem 6.2.2 (iv)]{go}). Write $\gamma_i(G)$ for the $i$th term of the lower central series of a group $G$. We have

\begin{corollary} \label{prime}
Let $n$ be a positive integer and $G$ a finite group admitting an automorphism $\phi$ of prime order $p$ such that $(|G|,p)=1$ and $|x^G|\leq n$ for each $x\in C_G(\phi)$. Let $h=h(p)$. Then the order of the derived group of $\gamma_{h+1}(G)$ is $n$-bounded. 
\end{corollary}
\begin{proof} Let $H$ be the normal closure of $C_G(\phi)$. Theorem \ref{main} tells us that the order of $H'$ is $n$-bounded. Note that $\phi$ acts fixed-point-freely on the quotient $G/H$, whence $G/H$ is nilpotent of class at most $h$. Therefore $\gamma_{h+1}(G)\leq H$ and so the derived group of $\gamma_{h+1}(G)$ is contained in $H'$. The result follows.
\end{proof}

Observe that in the particular case where $p=2$ Corollary \ref{prime} states that $G''$, the second derived group of $G$, has $n$-bounded order. This is because $h(2)=1$, that is, a finite group admitting a fixed-point-free automorphism of order two is necessarily abelian.

Recall that any finite soluble group $G$ has a Sylow basis ---  a family of pairwise permutable Sylow $p_i$-subgroups $P_i$ of $G$, exactly one for each prime divisor of the order of $G$, and any two Sylow bases are conjugate. The system normalizer (also known as the basis normalizer) of such a Sylow basis in $G$ is the intersection $T=\bigcap _i N_G(P_i)$. If $G$ is a finite soluble group and $T$ is a system normalizer in $G$, then $G=\langle T^G\rangle$ (see \cite[Theorem 9.2.8]{Rob}). Therefore we deduce

\begin{corollary} \label{system}
Let $n$ be a positive integer, $G$ a finite soluble group having a system normalizer $T$ such that $|x^G|\leq n$ for each $x\in T$. Then the order of the derived group $G'$ is $n$-bounded. 
\end{corollary}

Given a finite group $G$, the soluble residual of $G$ is defined as the smallest normal subgroup $R$ with the property that the quotient $G/R$ is soluble.

\begin{corollary} \label{residual}
Let $n$ be a positive integer and $G$ a finite group in which $|x^G|\leq n$ for each $2$-element $x\in G$. Then the order of the soluble residual of $G$ is $n$-bounded. 
\end{corollary}
\begin{proof} If $G$ has odd order, then by the Feit-Thompson Theorem \cite{fetho} $G$ is soluble and so the soluble residual of $G$ is trivial. Therefore we assume that the order of $G$ is even. Let $S$ be a Sylow 2-subgroup of $G$, and let $H$ be the normal closure of $S$. Theorem \ref{main} tells us that the order of $H'$ is $n$-bounded. The Feit-Thompson Theorem shows that $G/H'$ is soluble. Hence, $H'$ contains the soluble residual of $G$ and the result follows.
\end{proof}

\section{ Proof of Theorem \ref{restricted}}

Recall that a group $G$ is said to have restricted centralizers if for each $g$ in $G$ the centralizer $C_G(g)$ either is finite or has finite index in $G$. Of course, in the case where $G$ is profinite this is equivalent to saying that $C_G(g)$ either is finite or open. Shalev showed in \cite{shalev} that a profinite group with restricted centralizers is virtually abelian. The goal of  this section is to establish Theorem \ref{restricted}:\bigskip

{\it Let $p$ be a prime and $G$ a profinite group in which the centralizer of each $p$-element is either finite or open. Then $G$ has a normal abelian pro-$p$ subgroup $N$ such that $G/N$ is virtually pro-$p'$.}
\bigskip

Throughout, by a subgroup of a profinite group we mean a closed subgroup.
As usual, we say that a group has certain property locally if every finitely generated subgroup has that property. The next lemma is almost obvious. We include the proof for the reader's convenience. Note that the lemma is no longer true if we drop the assumption that $G$ is residually finite (cf.\ the non-abelian semidirect product of the Pr\"ufer group $C_{2^\infty}$ by the group of order 2).

\begin{lemma}\label{114} 
Let $G$ be a locally nilpotent group containing an element with finite centralizer. Suppose that $G$ is residually finite. Then $G$ is finite.
\end{lemma}
\begin{proof} Choose $x\in G$ such that $C_G(x)$ is finite. Let $N$ be a normal subgroup of finite index such that $N\cap C_G(x)=1$. Assume that $N\neq1$ and let $1\neq y\in N$. The subgroup $\langle x,y\rangle$ is nilpotent and so the center of $\langle x,y\rangle$ has nontrivial intersection with $N$. This is a contradiction since $N\cap C_G(x)=1$. The result follows.
\end{proof}

The next lemma is  somewhat similar to Lemma 2.6 in Shalev \cite{shalev}.
\begin{lemma}\label{fcbfc} 
Let $G$ be a profinite group having a subgroup $K$ such that the conjugacy class $x^G$ is finite for each $x\in K$. Then there is an integer $n$ such that $|x^G|\leq n$ for each $x\in K$.
\end{lemma}
\begin{proof} For each integer $i\geq1$ set $S_i=\{x\in K;\ |x^G|\leq i\}$. The sets $S_i$ are closed. Thus, we have at most countably many sets which cover the subgroup $K$. The Baire Category Theorem \cite[Theorem~34]{kel} says that at least one of these sets has non-empty interior. It follows that there exists an open normal subgroup $N$ in $G$, an element $a\in K$, and a positive integer $m$ such that $|y^G|\leq m$ whenever $y\in a(N\cap K)$. Set $K_0=N\cap K$. Since every element  $x$ of $K_0$ can be written  as a product of $a^{-1}$ and $ax$ and since the centralizers of both $a^{-1}$ and $ax$ have indices at most $m$, we conclude that $|x^G|\leq m^2$ whenever $x\in K_0$. Let $x_1,\dots,x_s$ be a full system of coset representatives for $K_0$ in $K$ and set $m_j=|{x_j}^G|$ for $j=1,\ldots, s$. Let $n$ be the product $m_1\cdots m_sm^2$. We deduce that $|x^G|\leq n$ for each $x\in K$, as required.
\end{proof}

We can now proceed with the proof of Theorem \ref{restricted}.
\begin{proof}[Proof of Theorem \ref{restricted}] Let $P$ be a Sylow $p$-subgroup of $G$. Suppose first that the conjugacy class $x^G$ is finite whenever $x$ is a $p$-element in $G$. By Lemma \ref{fcbfc} there is an integer $n$ such that $|x^G|\leq n$ for each $x\in P$. Let $H=\langle P^G\rangle$. Theorem \ref{main} tells us that the derived group $H'$ is finite. Choose an open normal subgroup $L$ in $G$ such that $L\cap H'=1$. Then $L\cap H$ is a normal abelian pro-$p$ subgroup while $G/(L\cap H)$ is virtually pro-$p'$. This proves the theorem in the case where the conjugacy classes of $p$-elements are finite.

Therefore we assume that $P$ has an element whose centralizer in $G$ is finite. Suppose that $P$ is torsion. Zelmanov's theorem \cite{zelmanov} says that $P$ is locally nilpotent. In view of Lemma \ref{114} the Sylow subgroup $P$ is finite and so $G$ is virtually pro-$p'$, as required.

Hence, we assume that $P$ contains an element $a$ of infinite order. This implies  that $C_G(a)$ is open, because  $\langle a\rangle$ is infinite and contained in $C_G(a)$. Let $C$ be an open normal subgroup of $G$ contained in $C_G(a)$. Since every element of $C$ centralizes $\langle a\rangle$, it follows that the centralizer in $G$ of every element from $P_0=P\cap C$ is open. By Lemma \ref{fcbfc} there is an integer $m$ such that $|x^G|\leq m$ for each $x\in P_0$. Let $M=\langle {P_0}^G\rangle$. In virtue of Theorem \ref{main} the derived group $M'$ is finite. Choose an open normal subgroup $J$ in $G$ such that $J\cap M'=1$. Then $J\cap M$ is a normal abelian pro-$p$ subgroup while $G/(J\cap M)$ is virtually pro-$p'$. The proof is complete.
\end{proof}

\end{document}